\documentclass[10pt,oneside]{article}
\usepackage{a4}
\usepackage{mathrsfs}
\usepackage{amsmath}
\usepackage{amssymb}
\usepackage{amsfonts}
\usepackage{graphicx}
\usepackage{color}
\newtheorem{thm}{Theorem}
\newtheorem{lem}[thm]{Lemma}

\newtheorem{cor}[thm]{Corollary}
\newtheorem{df}{Definition}
\newtheorem{obsn}{Observation}

\newtheorem{ex}{Example}
\newcommand{\bdf}{\begin{df} \begin{rm}}
\newcommand{\edf}{\end{rm} \end{df}}

\newenvironment{proof}{{\bf Proof.}}{\hspace*{\fill} \rule{2mm}{2mm} \par \hspace{0.1mm}}

\title{Super dominating sets in graphs}
\author{ M. Lema\'nska $^{1}$, V. Swaminathan $^{2}$
 Y.B. Venkatakrishnan$^{3}$,  R. Zuazua $^{4}$
\\
\\
$^1${\small  Gdansk University of Technology, Poland,  magda\@@mif.pg.gda.pl}
\\
$^2${\small  Saraswathi
Narayanan College, Madurai, India, sulanesri@yahoo.com}
\\
$^3${\small   SASTRA University, Tanjore, India,venkatakrish2\@@maths.sastra.edu }
\\
$^4${\small  Universidad Nacional Aut\'onoma de M\'exico, Mexico,  ritazuazua\@@gmail.com}
\\
}
\date{}
\begin{document}

\maketitle

\begin{abstract}
Let $G=(V,E)$ be a graph. A subset $D$ of $V(G)$ is called a super dominating set if for every $v \in V(G)-D$ there exists  an external private neighbour of $v$ with respect to $V(G)-D.$ The minimum cardinality of a super dominating set is called the super domination number of $G$ and is denoted by $\gamma_{sp}(G)$.  In this paper some results on the super domination number are obtained. We prove that if $T$ is a tree with at least three vertices, then $\frac{n}{2}\leq\gamma_{sp}(T)\leq n-s,$ where $s$ is the number of support vertices in $T$ and we  characterize the extremal trees.


\end{abstract}

\section{Introduction}  In this work we consider finite, undirected, simple graphs $G=(V,E)$ with $n$ vertices and $q$ edges.
The \emph{neighbourhood} of a vertex $v\in V(G)$ is the set $N_{G}(v)$ of all the vertices adjacent to $v$ in $G$.
For a set $X\subseteq V(G),$ \emph{the open neighbourhood} $N_{G}(X)$ is defined to be $\bigcup_{v\in X}N_{G}(v)$ and
\emph{the closed neighbourhood} $N_{G}[X]=N_{G}(X)\cup X.$ 

The \emph{degree} $d_{G}(v)$ of a vertex $v\in V(G)$ is the number of
edges incident to $v;$ $d_{G}(v)=|N_{G}(v)|.$  If $d_G(v)=n-1,$ then $v$ is a \emph{universal vertex} of $G$ and we call $v$ a \emph{semi-universal vertex} if every vertex of $G$ is either a neighbour of $v$ or has a common neighbour with $v.$ 

If $d_G(v)=1,$ then $v$ is an \emph{end-vertex} of $G.$ We denote by $\Omega(G)$ the set of end vertices of $G$. The neighbour of an end vertex is called a \emph{support vertex.} Support vertex which has more than one end-vertex as a neighbour is a \emph{strong support} vertex. For two sets $X, Y\subset V(G),$ we denote by $E(X,Y)$  the set of edges $uv\in E(G)$ such that $u\in X$ and $ v\in Y.$

For $X \subseteq  V (G)$  and $ x \in X,$  the set
$PN(x,X) =  N_G[x]- N_G[X -\{x\}]$
is called the \emph{private neighborhood} of $x$ with respect to $X$. 
An element $u$ of $PN(x, X)$ is called a \emph{private neighbour of x relative to X} or a $X$-private neighbour of $x.$ An  $X$-private neighbour of $x$ is either  $x$ itself, in wich case $x$ is an isolate vertex of $G[X]$, or is a neighbour of $x$ in $G$ which is not adjacent to any vertex of $X$. This latter type will be called an  $X$-\emph{external private neighbour}  of $x$.

A subset $D \subseteq V(G)$  is a \emph{dominating set} if for every vertex $u \in V(G)-D$, there exists $v \in D$ such that $u$ and $v$ are adjacent.  The minimum cardinality of a dominating set in $G$ is  the \emph{domination number}  of $G$ and is denoted by $\gamma(G)$.

It is well-known (Chapter 1,\cite{b4}) that a dominating set $D$ is minimal if and only if for any vertex $v \in  D$ there exists a private neighbourd of $v$ with respect $D.$  This fact was our motivation for the next definition.

\begin{df}\label{SDS}
A subset $D$ of $V(G)$ is called a super dominating set if for every vertex $v \in V(G)-D$ there exists an external private neighbour of $v$ with respect to $V(G)-D.$
\end{df}

The minimum cardinality of a super dominating set in $G$ is called
the \emph{super domination number} of $G$ and is denoted by
$\gamma_{sp}(G).$ A $\gamma_{sp}(G)-$set is a super dominating set
of $G$ with cardinality $\gamma_{sp}(G) .$ 
\begin{ex}
We can consider a super dominating set $D$ as a set of workers and $V(G)-D$ as a set of managers. Two workers (resp. manager) can be related or not. A worker has category A if it is adjacent to one and only one manager. By definition of superdominating set, each manager in $V(G)-D$ is adjacent at least one worker with category A.   \\

In the same way, we can consider a super dominating set $D$ as a set of students and $V(G)-D$ as a set of professors. Two students (resp. professors) can be related or not. A student is a Ph.D students if it has relation with one and only one professor. By the definition of superdomination set, each professor in  $V(G)-D$ has at least one Ph. D students.   \\ 

\end{ex}

The undefined terms in this paper may be found in \cite{b1}, \cite{b4}.

\begin{ex}
For $n\geq 2,$ the super domination number of the complete graph $K_n$ is $\gamma_{sp}(K_{n})=n-1$ and also for the star $K_{1,n-1}$ is $\gamma_{sp}(K_{1,n-1})=n-1$.\\

For a complete bipartite graph $K_{m,n}$ with $min\{ m,n\} \geq 2,$ the super domination number is $\gamma_{sp}(K_{m,n})=m+n-2 $.

\end{ex}

\begin{obsn} A set $D \subseteq  V(G)$ is a super dominating set of $G$ if and only if   for every $v \in V(G)-D$ there exists $u \in N_G(v) \cap D$ such that $N_G(u) \subseteq D \cup \{v\}$. 
\end{obsn}

\begin{df}\label{SDPS}
A super dominating set $D$ is a perfect set if $E(D, V(G)-D)$ is a perfect matching. 
\end{df}

\section{Preliminary results}

In this section we give some  observations about the bounds on the super domination number.

\begin{obsn}\label{LowerB}
For any graph $G$, $\gamma_{sp}(G)\geq \frac{n}{2}$ since
for  any $\gamma_{sp}(G)-$set $D,$ we have that $|D|\geq |V(G)-D|$.
\end{obsn}

The next observation follows immediately from the Definition $\ref{SDS}$.

\begin{obsn}\label{TrivialB}
For any graph $G$, the super domination number $\gamma_{sp}(G)=1$ if and only if $G\cong K_1$ or $G\cong K_2$ and $\gamma_{sp}(G)=n$ if and only if $G=\overline{K_n}.$
\end{obsn}

From the above  observations we have that for  any connected graph $G$ is $$\frac{n}{2}\leq \gamma_{sp}(G) \leq n-1.$$

The following theorem gives us a characterization of the connected graphs with $\gamma_{sp}(G)=\frac{n}{2}.$

\begin{thm}\label{twr1} For any connected graph $G$,  $\gamma_{sp}(G)=\frac{n}{2}$ if and only if every minimum super dominating set is a perfect set.

\end{thm}

\begin{proof}   Let $D$ be a $\gamma_{sp}-$set of $G$. If $D$ is a perfect set, then $|D|=|V(G)-D|$ and $\gamma_{sp}(G)=\frac{n}{2}.$ \\

Conversely, suppose $\gamma_{sp}(G)=\frac{n}{2}.$ Let $D$ be a minimum super dominating set of $G.$ If there exists a vertex $z\in D$ such that $|N(z)\cap (V(G)-D)|>1,$ then  from  Definition \ref{SDS} we obtain $|D|>|V(G)-D|,$ a contradiction. If there exists a vertex $z\in D$ such that $|N(z)\cap (V(G)-D)|=0$, then, since $|D|=|V(G)-D|,$ there is a vertex $w\in D$ such that $w$ has more than one neighbour in $V(G)-D,$ a contradiction. Thus every vertex of $D$ has exactly one neighbour in $V(G)-D.$

Suppose there is a vertex $a\in V(G)-D$ such that $|N(a)\cap D|\geq 2.$ Then either $|D|>\frac{n}{2}$ or there exists a vertex from $D$ which has more than one neighbour in $V(G)-D,$ again a contradiction. Thus every vertex of $V(G)-D$ has exactly one neighbour in $D.$ By
Definition $\ref{SDPS},$  $D$ is a perfect set.
\end{proof}

The next lemma gives us a relation between the super domination number of a connected graph $G$ and its diameter. Recall that the diameter of a graph $G$, denoted by $diam (G)$,  is defined to be the maximum distance between any two vertices $x,y\in V(G)$.

\begin{lem}
If $G$ is a connected graph with $diam (G)\geq 3,$ then $\gamma_{sp}(G) \leq n-2.$
\end{lem}

\begin{proof}
Suppose $diam(G)=k$, $k\geq 3$, $d_G(x,y)=k$ and $\alpha =(x,w_1,\cdots, w_{k-1},y)$ is any $xy-path$ with minimum length. Then $V(G)-\{ x,y\} $ is a  super dominating set of $G$ and $\gamma_{sp}(G) \leq n-2.$

\end{proof}

\begin{cor}\label{Diam2}
By  Observation $\ref{TrivialB}$ and the above Lemma, for any connected graph $G,$ if $\gamma_{sp}(G) =n-1,$ then $diam(G)\leq 2.$

\end{cor}

In particular, if $\gamma_{sp}(G) =n-1$ and $G$ is a connected graph, then for any $x\in V(G)$, $x$ is a universal vertex or $x$ is a semi-universal vertex.\\

The converse of Corollary $\ref{Diam2}$ is not true.  Consider a graph $G$ with
$V(G)=\{x_1, x_2, ..., x_k, y_1, y_2,..., y_k, u\}$ and  $E(G)=\{ x_1y_1, x_2y_2,
..., x_ky_k\} \cup \{ux_i, uy_i\}$ for  $1\leq i\leq k$ and $k\geq 2$. Then $diam(G)=2$, but
$\{x_1, x_2, ....,x_k, u\}$ is a super dominating set of $G$, so $\gamma_{sp}(G)\leq k+1<2k.$\\

The next theorem describes the super domination number of  a cycle with $n$ vertices.

\begin{thm}
For a cycle $C_n$ is

$$\gamma_{sp}(C_{n})=\left\{ \begin{array}{cl}
\left\lceil \frac{n}{2}\right\rceil   & \textrm{if $n\equiv 0,3\ (mod 4)$} \\  
                                       &         \\
\left\lceil \frac{n+1}{2}\right\rceil  & \textrm{otherwise.}
\end{array} \right.$$

\end{thm}

\begin{proof} Let $V(G)=\{ u_1,u_2,...,u_n\}$ be the set of vertices of $C_4$. We consider the following four cases.

\begin{enumerate}
    \item Let $n \equiv 0(mod 4)$. Then $n=4k$ for some positive integer $k.$ Consider the set $D=\{ u_{2},u_{3},u_{6},u_{7},\cdots,u_{4k-2},u_{4k-1}\}$. Clearly, $D$ is a super dominating set with $|D|=\frac{n}{2}$. By  Observation $\ref{LowerB}$, $D$ is minimum and $\gamma_{sp}(C_{n})=\left\lceil \frac{n}{2}\right\rceil .$

    \item Let $n \equiv 1 (mod 4)$. Then $n=4k+1$ for some positive integer $k.$ If $n=5,$ then it is easy to check that $\gamma_{sp}(C_5)=3=\frac{6}{2}=\frac{n+1}{2}.$ Let $n>5$ and consider $D=\{ u_{2},u_{3},u_{6},u_{7},\cdots,u_{4k-2},u_{4k-1}, u_{4k+1}\}$. Then $D$ is a  super dominating set with $|D|=\frac{n+1}{2}$. By  Observation  $\ref{LowerB}$, $D$ is minimum and $\gamma_{sp}(C_{n})=\left\lceil \frac{n+1}{2}\right\rceil $.

    \item Let $n \equiv 2 (mod 4)$. Then $n=4k+2$ for some positive integer $k.$ Consider $D=\{ u_{2},u_{3},u_{6},u_{7},\cdots,u_{4k-1},u_{4k+2}\}$ . Then $D$ is a  super dominating set with $|D|=2k+2$. By Def \ref{SDS}, it is not possible to have a super dominating set with $2k+1$ elements. Therefore, $D$ is minimum and $\gamma_{sp}(C_{n})=\left\lceil \frac{n+1}{2}\right\rceil $.

\item Let $n \equiv 3 (mod 4)$. Then $n=4k+3 $ for some positive integer $k.$ Consider $D=\{u_{2},u_{3},u_{6},u_{7},\cdots,u_{4k+2},u_{4k+3}\}$. Clearly, $D$
is a super dominating set with $|D|=\frac{n+1}{2}$. By Observation $\ref{LowerB}$, $D$ is minimum and $\gamma_{sp}(C_{n})=\left\lceil \frac{n}{2}\right\rceil .$
\end{enumerate}
\end{proof}

\begin{cor}
For a path $P_n$ with $n\geq 3$, $\gamma_{sp}(P_{n})=\left\lceil \frac{n}{2}\right\rceil.$
\end{cor}

Now we present a lemma showing the relation between the super domination number of $G$ and the number of its edges.

\begin{lem}\label{t1} For any connected graph $G$ with $ n > 1$, $\gamma_{sp}(G) \leq 2q - n + 1$. If $\gamma_{sp}(G)=2q-n+1$, then $G$ is a tree.
\end{lem}

\begin{proof} For any connected graph $G$, $\gamma_{sp}(G) \leq n-1 = 2(n-1)-n+1 \leq 2q-n+1$.  If $\gamma_{sp}(G)=2q-n+1$, then $q = n-1$ and $G$ is a tree.
\end{proof}

\begin{thm} For any graph $G$, $\gamma_{sp}(G) \geq n - \frac{1}{2}-\sqrt{\frac{2n^2-2n-4q+1}{4}}$ and the bound is sharp.
\end{thm}

\begin{proof} Let $D$ be a $\gamma_{sp}-$set of $G$.  Since $D$ is a super dominating set, for every $u \in V(G)-D$ there exists $v \in N(u) \cap D$ such that $N(v) \subseteq D \cup \{u\}$.  For every $u \in V(G)-D$, we can find an element $v \in D$ adjacent only to $u$.  Hence, $v$ is not adjacent to $n-\gamma_{sp}-1$ elements of $V(G)-D$.  Since there are $n-\gamma_{sp}$ elements in $V(G)-D$, we can find $n-\gamma_{sp}$ vertices in $D$ such that each of  $n-\gamma_{sp}$ vertices in $D$ is not adjacent to $n-\gamma_{sp}-1$ elements in $V(G)-D$. Therefore
\begin{center}
$q \leq \frac{n(n-1)}{2} - (n-\gamma_{sp}(G))(n-\gamma_{sp}(G)-1)$
\end{center}
and after some calculations we obtain
\begin{center}
 $\gamma_{sp}(G) \geq n - \frac{1}{2}-\sqrt{\frac{2n^2-2n-4q+1}{4}}$.
\end{center}
The bound is attained for the graph $C_{4}$.
\end{proof}

We finish this section with the following lemma that gives us a relation between the super domination number of a graph $G$ and its complement $\overline{G}.$ In $1956$ the original paper \cite{NG} by Nordhaus and Gaddum appeared. In it they gave
sharp bounds on the sum and product of the chromatic numbers of a graph and its
complement. Since then such results have been given for several parameters; see
for example \cite{NG2}. Here we have similar inequalities for the super domination number.

\begin{thm} For any graph $G,$ $n\leq \gamma_{sp}(G)+\gamma_{sp}(\overline{G})\leq 2n-1.$ The equality of the upper bound holds if and only if  $\{G,\overline{G}\}\cong K_n.$
\end{thm}

\begin{proof} The lower bound follows from Observation $\ref{LowerB}$.  Since $G$ and $\overline{G}$ can not be simultaneously  isomorphic to $\overline{K_n},$ by  Observation $\ref{TrivialB}$ we have the upper bound.

If $G\cong K_n,$ then $\overline{G}\cong \overline{K_n}$ and  $\gamma_{sp}(G)+\gamma_{sp}(\overline{G})= 2n-1.$ Assume now   $\gamma_{sp}(G)+\gamma_{sp}(\overline{G})= 2n-1.$ Then, without loss of generality, $\gamma_{sp}(G)=n$ and $\gamma_{sp}(\overline{G})=n-1.$ From the Observation $\ref{TrivialB}$, $G\cong \overline{K_n}$ and $\overline{G} \cong K_n .$

The lower bound is also attained, for example if $G\cong P_4.$
\end{proof}

\section{Super domination in trees}

In this section we are interested in giving a characterization of  trees which attain the lower and the upper bounds in terms of the number of vertices. The following Theorem gives us an upper bound for  trees. 

\begin{thm}\label{twr2} If $T$ is a tree with at least three vertices, then $\frac{n}{2}\leq\gamma_{sp}(T)\leq n-s,$ where $s$ is the number of support vertices in $T.$
\end{thm}

\begin{proof}
The lower bound follows from Observation \ref{LowerB}. Let $Supp$ be the set of supports in $T,$ $|Supp|=s.$ Let $S_1\subseteq \Omega(T),$ $S_1=\{v_1,\ldots,v_{|Supp|}\}$ such that every $v_i, 1\leq i \leq |Supp|$ is adjacent to a different vertex from $Supp.$ Then $V-S_1$ is a super dominating set of $T$ and $\gamma_{sp}(T)\leq n-s.$
\end{proof}

Now we characterize the extremal trees of Theorem \ref{twr2}. We begin with the following observation.

\begin{obsn}\label{drzewa} Let $T$ be a tree such that $V(T)=2m$ and $\gamma_{sp}=m$ for some $m\geq 1.$ If there are two trees $T_1, T_2$ with $V(T_1)\cup V(T_2)=V(T), V(T_1)\cap V(T_2)=\emptyset$ and $E(T_1), E(T_2)\subseteq E(T)$ such that $\gamma_{sp}(T)=\gamma_{sp}(T_1)+\gamma_{sp}(T_2),$ then $\gamma_{sp}(T_1)=\frac{|V(T_1)|}{2}$ and $\gamma_{sp}(T_2)=\frac{|V(T_2)|}{2}.$
\end{obsn}
\begin{proof} Without loss of generality, suppose $\gamma_{sp}(T_1)>\frac{|V(T_1)|}{2}.$ Then $\gamma_{sp}(T)=\gamma_{sp}(T_1)+\gamma_{sp}(T_2)>\frac{|V(T_1)|}{2}+\frac{|V(T_2)|}{2}>m.$
\end{proof}

Let $\mathcal{R}$ be the family of trees $T$ that can be obtained from a sequence $T_1,\ldots,T_j$ ($j\geq 1$) of trees such that:\\

1. The tree $T_1=P_2=(a_1,b_1).$\\

2. If $j\geq 2,$ define the tree $T_j$ such that $V(T_j)=V(T_{j-1})\cup \{a_j,b_j\}$ and $E(T_j)=E(T_{j-1})\cup \{a_jb_j\}\cup \{e\},$ where $e=a_ia_j$ or $e=b_ib_j$ for some $1\leq i\leq j-1.$\\

If $T\in \mathcal{R}$ with $|V(T)|=2m,$ denote by $A=\{a_1,\ldots, a_m\}$ and $B=\{b_1,\ldots,b_m\}.$

Observe that for all $1\leq i\leq m,$ $N_T(a_i)\cap B=\{b_i\}$ and $N_T(b_i)\cap A=\{a_i\}.$

It is clear that $A$ and $B$ are minimum super dominating sets of $T.$ Moreover, by construction of $T,$ $T$ does not have strong support vertices.

We  say that a vertex $v\in A$ has a status $a$ ($sta(v)=a$) and a vertex $v\in B$ has a status $b$ ($sta(v)=b$).

\begin{thm} Let $T$ be a tree with $|V(T)|=n\geq 2.$ Then $\gamma_{sp}(T)=\frac{n}{2}$ if and only if $T\in \mathcal{R}.$
\end{thm}
\begin{proof} If $T\in \mathcal{R},$ then by definition of the family $\mathcal{R},$  $\gamma_{sp}(T)=\frac{n}{2}.$
Let $T$ be a tree with  $\gamma_{sp}(T)=\frac{n}{2};$ then $n=2m$ for $m\geq 1$. We show that $T\in \mathcal{R}$ by induction on $m.$  If $m=1,$ then $T=P_2=(a_1,b_1)\in \mathcal{R}.$ If $m=2$ and $\gamma_{sp}(T)=2,$ then $T=P_4=(a_1, b_1, b_2, a_2)\in \mathcal{R}.$ Assume that $|V(T)|=2m\geq 6$ and the Theorem holds for any $T'$ with $|V(T')|=2m'$ with $m'<m.$

Let $P=(x_1,\ldots,x_p)$ be a longest path of $T$ and let $D$ be a minimum super dominating set of $T.$ By definition of $D,$ $|\{x_1,x_2\}\cap D|\geq 1.$ We consider the following cases. \\

$Case$ $1.$ The vertices $x_1, x_2\in D.$ Let $T'=T-\{x_2x_3\}=T_1\cup T_2,$ where $x_2\in T_1,$ $x_3\in T_2.$ Since $P$ is a longest path of $T,$ $T_1$ is a star and $\gamma_{sp}(T_1)=|V(T_1)|-1.$ \\

$a)$ Suppose $x_3\in D.$

 Then $d_T(x_2)>2$ (in the other case $D-\{x_1\}$ is a super dominating set of $T,$ a contradiction).  We have  $\gamma_{sp}(T_1)=|V(T_1)|-1>\frac{|V(T_1)|}{2}.$ Since  $\gamma_{sp}(T)=\gamma_{sp}(T_1)+\gamma_{sp}(T_2),$ we have a contradiction with Observation \ref{drzewa}.

$b)$ Suppose $x_3\notin D.$

Observe that $\Omega(T)\cap N_T(x_2)\subseteq D.$ If $N_T(x_3)\cap
D\neq \{x_2\},$ then $D-\{x_2\}$ is a super dominating set of $T.$
 Therefore $x_2$ is
the only neighbour of $x_3$ belonging to $D$ and $x_4\notin D.$

Let $T'=T-\{x_3x_4\}=T_1\cup T_2,$ where $x_3\in T_1,$ $x_4\in T_2.$ Since $x_3, x_4\notin D,$ $\gamma_{sp}(T)=\gamma_{sp}(T_1)+\gamma_{sp}(T_2).$

If $T_1$ is a star, then we have the same contradiction as above.
Thus $T_1$ is not a star. Since $N_T(x_3)\cap D=\{x_2\},$ there
exists $v\in \Omega(T)$ such that $d_T(v,x_3)=2$ and $v\in D.$
Similarly to the case above,
$\gamma_{sp}(T_1)>\frac{|V(T_1)|}{2},$ a contradiction.

$Case$ $2.$ Suppose $x_2\notin D.$
 If $d_T(x_2)>2,$ then $N_T(x_2)\cap \Omega(T)\subseteq D.$ Let $T'=T-{x_2x_3}=T_1\cup T_2,$ where $x_2\in T_1, x_3\in T_2.$  Then $T_1$ is a star with three or more vertices and $\gamma_{sp}(T_1)=|V(T_1)|-1>\frac{|V(T_1)|}{2}.$ 

Therefore $d_T(x_2)=2.$ Let $T^{''}=T-\{x_1, x_2\}.$ We have $|V(T^{''})|=n-2=2(m-1)$ and $\gamma_{sp}(T^{''})=\gamma_{sp}(T)-1=m-1.$ By the induction hypothesis, $T^{''}\in \mathcal{R}.$ Then $T=T^{''}\cup \{x_1, x_2, x_1x_2, x_2x_3\}\in \mathcal{R}$ with $sta(x_2)=sta(x_3)$ and $sta(x_1)\neq sta(x_2).$

$Case$ $3.$ Suppose $x_1\notin D.$ Then $x_2\in D.$ If
$d_T(x_2)>2,$ then there exists $z\in N_T(x_2)\cap \Omega(T)$ such that $z\neq x_1$ and $z\in D.$ Let $T^{'''}=T-zx_2=T_1\cup T_2, T_1=\{z\}.$ Then we have $\gamma_{sp}(T)=\gamma_{sp}(T_1)+\gamma_{sp}(T_2)\geq 1+\frac{n(T_2)}{2}=\frac{n+1}{2}>\frac{n}{2},$ a contradiction.

 Therefore
$d_T(x_2)=2$ and the proof $T\in \mathcal{R}$ is the same as in
the above case.
\end{proof}

Now we are in position to characterize all trees for which the upper bound is attained.
Let  $\mathcal{S}$ be the family of trees $T$ that can be obtained from a sequence $T_1,\ldots,T_j$ ($j\geq 1$) of trees such that $T_1=P_3=(a_1,b_0,b_1)$ and if $j\geq 2,$ define the tree $T_j$ in the following three ways:\\

$1.$ $V(T_j)=V(T_{j-1})\cup \{a_l, b_k\}$ and $E(T_j)=E(T_{j-1})\cup \{a_lb_k\}\cup \{b_ib_k\}$ for some $b_i\in V(T_{j-1});$ or \\

$2.$ $V(T_j)=V(T_{j-1})\cup \{a_l\}$ and $E(T_j)=E(T_{j-1})\cup \{a_lb_i\}$ for some non-support vertex  $b_i\in V(T_{j-1});$ or \\

$3.$  $V(T_j)=V(T_{j-1})\cup \{ b_k\}$ and $E(T_j)=E(T_{j-1})\cup \{b_kb_i\}$ for some support vertex  $b_i\in V(T_{j-1}).$ \\

If $T\in \mathcal{S},$  denote by $A=\{a_1,\ldots, a_l\}$ and $B=\{b_1,\ldots,b_k\},$ where $l+k=n.$

Similarly like in the family $\mathcal{R},$ we will say that a vertex $v\in A$ has a status $a$ ($sta(v)=a$) and a vertex $v\in B$ has a status $b$ ($sta(v)=b$).

Observe that for all $1\leq i, |N_T(a_i)\cap B|=1$ and $|N_T(b_i)\cap A|\leq 1.$

By definition of the family $\mathcal{S},$ it is clear that the
set of vertices $B$ is a minimum super dominating set of $T$ and
$|A|=|Sp|,$ where $Sp$ is the set of supports in $T$. Therefore
$|B|=|V-A|=|V-Sp|=n-s.$ Moreover, if $v$ is a non-end vertex  of
$T,$ then $sta(v)=b.$

\begin{thm} If $T$ is a tree with at least three vertices and $s$ is the number of supports in $T,$ then $\gamma_{sp}(T)=n-s$ if and only if $T\in \mathcal{S}.$
\end{thm}
\begin{proof} By definition of the family $\mathcal{S},$ if $T\in \mathcal{S},$ then $\gamma_{sp}(T)=n-s.$

Assume $T$ is a tree with $|V(T)|=n$ and $\gamma_{sp}(T)=n-s.$ We
show that $T\in \mathcal{S}$ by induction on $n.$

If $n=3,$ then $T=P_3=(a_1, b_0, b_1)$ and $T\in \mathcal{S}.$ If $n=4,$ then $T=P_4$ or $T=K_{1,3};$ in both cases $T\in \mathcal{S}.$

Assume $n\geq 4$ and the Theorem holds for every $m<n.$

Let $P=(v_1,\ldots,v_p)$ be a longest path in $T$ such that $d_T(v_2)$ is as big as possible. Let $D$ be a minimum super dominating set of $T.$ We consider the following cases.\\

$Case$ $1.$ Assume $d_T(v_2)>2.$ Let $T'=T-\{v_1\}.$ Then
$|V(T')|=n-1$ and $s(T')=s(T)=s.$ If $v_1\in D,$ then $D-\{v_1\}$
is a minimum  super dominating set of $T'$ and
$\gamma_{sp}(T')=\gamma_{sp}(T)-1=(n-1)-s.$ By the induction
hypothesis, $T'\in \mathcal{S}.$ Since $v_2$ is a non-end vertex
of $T',$ we have $sta(v_2)=b.$ Putting $sta(v_1)=b$ we have that
$T\in \mathcal{S}.$

Now, if $v_1\notin D,$ then $v_2\in D$ and since $d_T(v_2)>2,$ there exists a vertex $x\in N_T(x_2)\cap \Omega(T)$ such that $x\in D.$ Define $T'=T-\{x\}$ and similarly like in the above case we obtain $T\in \mathcal{S}.$\\

$Case$ $2.$ Suppose $d_T(v_2)=2.$ Here we have three cases.\\

$1.$ The set $\{v_1, v_2\}\subseteq D.$ Since $D$ is minimum, $v_3\notin D.$ If $v_3$ is a support vertex, then all of its end neighbours belong to $D$ and $D-\{v_2\}$ is a super dominating set of $T,$ a contradiction. If $v_3$ is not a support vertex, $d_T(v_3,\Omega(T))=2$ and also  $D-\{v_2\}$ is a super dominating set of $T,$ a contradiction.\\

$2.$ The vertex $v_1\in D$ and $v_2\notin D.$ Let $T'=T-\{v_1,v_2\}.$ We have $\gamma_{sp}(T')=\gamma_{sp}(T)-1,$ $|V(T')|=|V(T)|-2, s(T')=s-1.$ Then $\gamma_{sp}(T')=|V(T')|-s(T').$ From the induction hypothesis, $T'\in \mathcal{S}.$ We can obtain $T$ from $T',$ putting $sta(v_1)=a, sta(v_2)=b$ and thus $T\in \mathcal{S}.$\\

$3.$ The vertex $v_1\notin D$ and $v_2\in D.$ By definition of $D,$ $v_3\in D.$ Similarly to the previous case we can consider $T'=T-\{v_1, v_2\}$ and prove that $T\in \mathcal{S}.$
\end{proof}

\end{document}